\documentclass[11pt,a4paper]{amsart}
\usepackage{amsmath,amsfonts,amssymb,amscd,amsthm,amsrefs,changes}

\usepackage[a4paper]{geometry}
\usepackage[shortlabels]{enumitem}

\usepackage{color, soul}
\definecolor{gr}{rgb}{0.7, 1, 0.7}
\definecolor{rr}{rgb}{1, 0.7, 0.7}

\usepackage{latexsym}

\usepackage{graphicx}

\usepackage{hyperref}

\theoremstyle{plain} 
\newtheorem{theorem}{Theorem}[section]

\newtheorem{lemma}[theorem]{Lemma}
\newtheorem{corollary}[theorem]{Corollary}

\newtheorem{proposition}[theorem]{Proposition}

\theoremstyle{definition} 
\newtheorem{definition}[theorem]{Definition}
\theoremstyle{remark} 

\renewcommand{\mathfrak}{\mathbf}

\newcommand{\ignore}[1]{}

\newcommand{\bbC}{\mathbb{C}}
\newcommand{\bbP}{\mathbb{P}}

\newcommand{\bbZ}{\mathbb{Z}}

\newcommand{\bfa}{\mathbf{a}}
\newcommand{\bfb}{\mathbf{b}}

\newcommand{\bfm}{\mathbf{m}}

\newcommand{\Rat}{\mathrm{Rat}}
\newcommand{\PSL}{\mathrm{PSL}}

\newcommand{\Cb}{\mathbb{C}}

\title[]{Independence of multipliers via degenerations of rational maps}

\author{Igors Gorbovickis}
\thanks{The research of the author was supported by the German Research Foundation (DFG, project number 455038303).}
\address{Constructor University, Bremen, Campus Ring 1, 28759 Bremen, Germany}
\email{igorbovickis@constructor.university}

\date{\today}

\begin{document}
	
	\begin{abstract}
		
		Using degenerations of rational maps and local asymptotics of periodic multipliers, we prove that for every $d\ge 2$, the multipliers of any $2d-2$ distinct periodic orbits of degree $d$ rational maps are algebraically independent over $\bbC$, provided that at most $d$ of the selected orbits are fixed points.
		This condition is sharp because of the Holomorphic Index Formula that relates the multipliers of the $d+1$ fixed points. 
		The result of this paper removes the additional restrictions on periods present in earlier work. 
		The proof proceeds by induction on the degree, using a one-hole degeneration for the induction step and a three-hole degeneration in an exceptional degree four case.”		
	\end{abstract}
	
	\maketitle

	\section{Introduction}
	
	For $d\ge 1$, let $\Rat_d$ be the space of degree $d$ rational maps of the Riemann sphere. There is a natural injective map from $\Rat_d$ to the $(2d+1)$-dimensional complex projective space $\bbP^{2d+1}$ defined by taking
	$$
	f(z) = \frac{p(z)}{q(z)} = \frac{a_dz^d + \ldots + a_0}{b_dz^d+\ldots+b_0}
	$$
	to $(a_0\colon\ldots\colon a_d\colon b_0\colon\ldots\colon b_d)\in \bbP^{2d+1}$. 
	The image of this map is the complement of a certain codimension $1$ algebraic subset, defined as the zero locus of the resultant of $p$ and $q$. Thus, $\Rat_d$ has a structure of an irreducible quasiprojective variety.
	The group $\PSL_2(\bbC)$ acts on $\Rat_d$ by conjugation.  For $d\ge 2$, the corresponding moduli space
	\[
	\mathcal M_d:=\Rat_d/\PSL_2(\bbC)
	\]
	is an affine algebraic variety of dimension $2d-2$; see \cite{Silverman1998} and \cite[Chapter~4]{SilvermanADS}.
	
	A key point in studying the moduli spaces $\mathcal M_d$ is the choice of a parameterization.  The use of multipliers of periodic orbits as local parameters has proved fruitful in several settings.  For instance, Milnor used the fixed-point multipliers to show that $\mathcal M_2$ is biholomorphic to $\bbC^2$ \cite{Milnor1993}. For higher degree maps and arbitrary periods, the multiplier coordinates have been studied in~\cite{McMullen1987},~\cite{Gorbovickis2015},~\cite{Gorbovickis2016},~\cite{JiXie2025},~\cite{Huguin}. Further generalizations for polynomial endomorphisms of $\bbC^n$ and endomorphisms of $\bbP^n$ were obtained in~\cite{GTV_sparsity} and~\cite{GT2025}.

	Let $f\in\Rat_d$, and let $\mathcal O$ be a periodic orbit of $f$ of period $m$ whose multiplier is different from $1$. (Here and below, the period of a periodic orbit $\mathcal O$ is the cardinality of $\mathcal O$, viewed as a finite set, and the multiplier 
	is defined as 
	$$
	\lambda_{\mathcal O}(f):=(f^m)'(z),
	$$
	for $z\in\mathcal O$.) Then, according to the implicit function theorem, in a sufficiently small neighborhood of $f$ in $\Rat_d$, the orbit $\mathcal O$ admits a holomorphic continuation, and its multiplier defines a holomorphic local function. We shall call this function the local multiplier function corresponding to $\mathcal O$ and denote it by $\lambda_{\mathcal O}$. The local multiplier function further extends to a global multiple-valued algebraic function on the whole of $\Rat_d$. Since the multipliers are invariant under M\"obius conjugacy, the (global) multiplier functions also project to multiple-valued algebraic functions on $\mathcal M_d$.

	Notably, McMullen proved that, apart from flexible Latt\`es maps, an element of the moduli space $\mathcal M_d$ is determined up to finitely many choices by the multipliers of all of its periodic orbits \cite{McMullen1987}. 
	In particular, one can select $2d-2=\dim\mathcal M_d$ distinct periodic orbits whose multipliers are algebraically independent and hence form local coordinates at a generic point of $\mathcal M_d$.  This leaves the more precise question of which collections of $2d-2$ periodic-orbit multipliers can be used as local parameters.  Equivalently, do some such collections satisfy ``hidden'' algebraic relations?  A classical example of such a relation is provided by the Holomorphic Index Formula.  Namely, if $f\in\Rat_d$ has $d+1$ distinct fixed points $z_1,\ldots,z_{d+1}$ with multipliers $\lambda_i=f'(z_i)$, then
	\[
	\sum_{i=1}^{d+1}\frac{1}{1-\lambda_i}=1.
	\]
	Thus, the multipliers of all fixed points can never be algebraically independent.

	In this paper we show that, among collections of $2d-2$ distinct periodic orbit multipliers, the Holomorphic Index Formula gives the only ``hidden'' algebraic relation of this type. More specifically, we prove the following:

	\begin{theorem}\label{main_theorem}
		For any integer $d\ge 2$, the multipliers of any $2d-2$ distinct periodic orbits, viewed as multiple-valued algebraic functions on $\Rat_d$, are algebraically independent over $\bbC$, provided that no more than $d$ of these orbits have period $1$.
	\end{theorem}
	This is the sharpest possible result, because if the condition of the theorem is violated and $d+1$ fixed points are selected among the $2d-2$ distinct periodic orbits, then the multipliers of these periodic orbits cannot be algebraically independent precisely due to the Holomorphic Index Formula.
	
	Theorem~\ref{main_theorem} strengthens the previous result from~\cite{Gorbovickis2015}, where the following was shown:
	
	\begin{theorem}\label{old_theorem}\cite{Gorbovickis2015}
		For any integer $d\ge 2$, the multipliers of any $2d-2$ distinct periodic orbits, viewed as multiple-valued algebraic functions on $\Rat_d$, are algebraically independent over $\bbC$, provided that either $d=2$ or $d\ge 3$ and the following conditions hold: if $m_1,\ldots,m_{2d-2}$ are the periods of the selected orbits, then
		\begin{enumerate}[(i)]
			\item\label{cond1} no more than $d$ of the periods $m_1,\ldots,m_{2d-2}$ are equal to $1$, and
			\item\label{cond2} at least one of the periods $m_1,\ldots,m_{2d-2}$ is greater than $2$ while the remaining $2d-3$ periods are not simultaneously equal to $2$.
		\end{enumerate}
	\end{theorem}
	Note that for $d=2$, condition~\ref{cond1} (which is also present in Theorem~\ref{main_theorem}) holds automatically, so Theorem~\ref{main_theorem} essentially shows that for $d\ge 3$, condition~\ref{cond2} in Theorem~\ref{old_theorem} can be omitted.

	The main new ingredient in the proof is a degeneration argument. 
	Degenerations of rational maps and their rescaling limits have been studied extensively; see, for example, DeMarco~\cite{DeMarco2005}, Kiwi~\cite{Kiwi2015}, Luo~\cite{Luo_22}, and the references therein. More recently, periodic multipliers under degeneration have been studied by Favre~\cite{Favre_25}, Gong~\cite{Gong_25} and Favre-Gong~\cite{Favre_Gong_26}; see also Huguin~\cite{Huguin} in the polynomial setting. Our use of degeneration is somewhat different: we exploit explicit local multiplier asymptotics near holes of the limiting degenerate map in order to prove transversality of multiplier functions.

	As an immediate consequence of Theorem~\ref{main_theorem}, we obtain hyperbolic components of disjoint type with prescribed periods of the attracting cycles. Recall that a hyperbolic component is of \emph{disjoint type} if, for every map in the component, all $2d-2$ critical points are simple and are attracted to pairwise distinct attracting cycles.
	Note that each attracting orbit attracts a critical point, while a degree $d$ rational map has $2d-2$ critical points counted with multiplicity.  Consequently, a degree $d$ rational map cannot have more than $2d-2$ distinct attracting cycles. Furthermore, a degree $d$ rational map belongs to a hyperbolic component of disjoint type if and only if it has exactly $2d-2$ distinct attracting cycles.

	\begin{corollary}\label{cor_hyperbolic_components}
		Let $d\ge 2$, and let $(m_1,\ldots,m_{2d-2})$ be a tuple of positive integers for which a degree $d$ rational map can have $2d-2$ distinct periodic orbits of respective exact periods $m_1,\ldots,m_{2d-2}$.  If at most $d$ of the integers $m_i$ are equal to $1$, then $\mathcal M_d$ contains a hyperbolic component of disjoint type characterized by $2d-2$ distinct attracting periodic orbits of respective exact periods $m_1,\ldots,m_{2d-2}$.
	\end{corollary}
	
	\begin{proof}
		On the algebraic cover of $\Rat_d$ obtained by marking the selected periodic orbits, their multipliers define an algebraic map to $\bbC^{2d-2}$.  By Theorem~\ref{main_theorem}, this map is dominant, and hence its image contains a nonempty Zariski open subset of $\bbC^{2d-2}$. 
		This subset intersects the unit polydisk, so there exists a map $f\in\Rat_d$ with $2d-2$ distinct attracting periodic orbits of the prescribed periods.  
	\end{proof}
	
	\textbf{AI use declaration.} During the preparation of this paper, the author used ChatGPT (OpenAI) for exploratory mathematical discussions, checking computations in Sections~\ref{sec:added_pole} and~\ref{sec:lemma_proof}, and revising parts of the exposition. An early version of the three-hole degeneration used in Section~\ref{sec:lemma_proof} arose from such an interaction and was subsequently substantially developed and refined by the author through further analysis and computation. The author has verified the mathematical arguments and takes full responsibility for the contents of the paper.
	
	\textbf{Acknowledgments.} The author thanks Valentin Huguin for several stimulating discussions. In connection with a different problem, he shared with the author the idea of studying multiplier maps at near-degenerate rational maps; this idea subsequently inspired the approach developed in the present paper.

	\section{Strategy of the proof of Theorem~\ref{main_theorem}}

	\subsection{Marked periodic orbits and the one-hole induction step}
	
	The multipliers as functions are naturally defined on the space of rational maps with marked periodic orbits. That is, for a fixed integer $d\ge 2$, and a tuple of positive integers $\bfm=(m_1,\ldots, m_{2d-2})\in\bbZ_{>0}^{2d-2}$, let 
	$$
	X_d^\bfm\subset\Rat_d\times\widehat{\bbC}^{2d-2}
	$$
	be the closure of the set of all elements $(f,z_1,\ldots,z_{2d-2})\in \Rat_d\times\widehat{\bbC}^{2d-2}$ such that $z_j$ is a periodic point of $f$ of (minimal) period $m_j$, for each $j\in\{1,\ldots,2d-2\}$, and no two points $z_i$ and $z_j$ with $1\le i<j\le 2d-2$ belong to the same periodic orbit.

	As a consequence of the irreducibility result of Lau and Schleicher~\cite{LauSchleicher1994} (see also~\cite{Morton1998} for a more algebraic proof), one can show that, for every $d\ge 2$ and every $\bfm\in \bbZ_{>0}^{2d-2}$ for which the defining locus is nonempty, the set $X_d^\bfm$ is an irreducible algebraic variety.  For a detailed proof, see~\cite{Gorbovickis2015}.  As a consequence, the following local-to-global criterion was obtained in~\cite{Gorbovickis2015}[Proposition~2.5].
	
	\begin{definition}
		We say that a finite collection of local multiplier functions is locally independent at a map $f\in\Rat_d$ if their differentials at $f$ are linearly independent.	
	\end{definition}

	\begin{proposition}\label{prop_local_to_global}[\cite{Gorbovickis2015}, Proposition~2.5]
		Let $d\ge 2$ and $\bfm=(m_1,\ldots,m_{2d-2})\in\bbZ_{>0}^{2d-2}$.  Suppose that there exist a map $g\in\Rat_d$ and $2d-2$ non-multiple periodic points $z_1,\ldots,z_{2d-2}$ of $g$, belonging to distinct periodic orbits of respective periods $m_1,\ldots,m_{2d-2}$, such that their local multiplier functions are locally independent 
		at $g$.  Then the multipliers of any $2d-2$ distinct periodic orbits of respective periods $m_1,\ldots,m_{2d-2}$, viewed as multiple-valued algebraic functions on $\Rat_d$, are algebraically independent over $\bbC$.
	\end{proposition}
	
	The proof of Theorem~\ref{old_theorem} was obtained in~\cite{Gorbovickis2015} by applying Proposition~\ref{prop_local_to_global} to the map $g(z) = z^d$. An important ingredient was a direct computation of the differentials of the multiplier functions at $g$ which was feasible due to the simple form of the map $g$. At the same time, the computation at $g(z) = z^d$ does not seem to be sufficient to get rid of the condition~\ref{cond2} in Theorem~\ref{old_theorem} and obtain the proof of Theorem~\ref{main_theorem}.
	
	In order to prove Theorem~\ref{main_theorem}, we apply Proposition~\ref{prop_local_to_global} to near-degenerate rational maps $g$. 
	The proof of Theorem~\ref{main_theorem} proceeds by induction on the degree.

	More specifically, for a rational map $f\colon \widehat{\Cb}\to \widehat{\Cb}$ of degree $d\ge 1$ on the Riemann sphere $\widehat{\Cb}$, consider the two-parameter family
	
	\begin{equation}\label{g_ab_eq}
		g_{a,b}(z)
		:=
		f(z)\frac{z-a+b}{z-a}
		=
		f(z)\left(1+\frac{b}{z-a}\right),
	\end{equation}
	where $a, b\in\bbC$ are the parameters. For generic parameters $a, b\in\bbC$, we have $\deg g_{a,b} = \deg f+1 = d+1$. For $b\neq 0$ sufficiently small, $g_{a,b}$ is a degree $d+1$ rational map approaching the boundary of $\Rat_{d+1}$ as $b\to 0$. Viewed in the projective compactification $\overline{Rat}_{d+1}$, the limiting map $g_{a,0}$ is degenerate: it has $a$ as a \textit{hole} of depth one and $f$ as its reduced map, in the terminology of DeMarco~\cite{DeMarco2005}. We will apply Proposition~\ref{prop_local_to_global} to maps $g_{a,b}$ with $b$ close to zero.

	The main new idea is to use degeneration to pass from degree $d$ to degree $d+1$. Starting with a degree $d$ map for which $2d-2$ multiplier differentials are independent, we introduce a small additional pole. As the pole parameter tends to zero, the original periodic orbits persist, while new fixed and period-two orbits appear near the hole. Their multipliers diverge at an explicitly computable rate. By varying the position and size of the hole, these new multipliers provide two additional independent cotangent directions. This yields the induction step from degree $d$ to degree $d+1$. 
	
	A precise formulation of the induction step is given in the following proposition:

	\begin{proposition}\label{prop_induction_step}
		For any integer $d\ge 2$, there exists a nonempty Zariski open subset $\mathcal V\subset \Rat_d\times\bbC$ with the following property:
		if $(f,a_0)\in \mathcal V$ and
		\[
		\mathcal O_1,\ldots,\mathcal O_{2d-2}
		\]
		are non-multiple periodic orbits of $f$ of respective periods
		\[
		m_1,\ldots,m_{2d-2}
		\]
		with the multipliers being locally independent at $f$ 
		and with 
		\[
		a_0\notin \mathcal O_k
		\qquad\text{for all } k=1,\ldots,2d-2,
		\]
		then there exists $b\in\bbC$ arbitrarily close to $0$, with $b\neq 0$, such that the map $g_{a_0,b}$ from~(\ref{g_ab_eq}) has locally independent multipliers of periods
		\[
		2,2,m_1,\ldots,m_{2d-2},
		\]
		and also locally independent multipliers of periods
		\[
		1,2,m_1,\ldots,m_{2d-2}.
		\]
	\end{proposition}
	
	There is one exceptional combinatorial case that cannot be reached by this induction, namely the case $d=4$ in which all six selected cycles have period two. We have to cover this case in a separate proposition:

	\begin{proposition}\label{lemma_period_two_degree_four}
		For $d=4$, the multipliers of the six period $2$ cycles, viewed as multiple-valued algebraic functions on $\Rat_4$, are algebraically independent over~$\bbC$.	
	\end{proposition}

	In the proof of Proposition~\ref{lemma_period_two_degree_four} we use a different degeneration. Namely, we degenerate a degree $4$ rational map to a M\"obius transformation with three holes. The six period-two cycles are then all created by the degeneration, and the corresponding multiplier Jacobian admits a natural limiting block decomposition.

	The remainder of the paper, Sections~\ref{sec:added_pole} and~\ref{sec:lemma_proof}, is devoted to 
	the proof of Propositions~\ref{prop_induction_step} and~\ref{lemma_period_two_degree_four}. In the following subsection we  give a proof of Theorem~\ref{main_theorem} assuming the results of Proposition~\ref{prop_induction_step} and Proposition~\ref{lemma_period_two_degree_four}.

	\subsection{Proof of the main theorem modulo Propositions~\ref{prop_induction_step} and~\ref{lemma_period_two_degree_four}}

	\begin{definition}
		For any $d\ge 2$, a tuple of $2d-2$ positive integers $\bfm=(m_1,\ldots,m_{2d-2})$ is called \textit{admissible} if a generic degree $d$ rational map admits $2d-2$ distinct periodic orbits of respective periods $m_1,\ldots, m_{2d-2}$ and $\bfm$ satisfies condition~\ref{cond1} of Theorem~\ref{old_theorem}.
	\end{definition}

	\begin{lemma}\label{lemma_admissible_reduction}
		Assume, for some $d\ge 3$, a tuple $\bfm=(m_1,\ldots, m_{2d-2})$ is admissible. Then at least one of the following possibilities holds:
		\begin{enumerate}[(a)]
			\item\label{cond_a} $\bfm$ satisfies condition~\ref{cond2} of Theorem~\ref{old_theorem};
			\item\label{cond_b} $\bfm$ contains a pair of numbers $2, 2$ or $1, 2$, such that after removing this pair of numbers from $\bfm$, one obtains an admissible tuple that corresponds to the degree $d-1$;
			\item\label{cond_c} $d=4$ and $m_1=m_2=\ldots=m_6 = 2$.
		\end{enumerate}
	\end{lemma}
	\begin{proof}
		For a generic rational map of degree $q\ge 2$, the number of fixed points is $q+1$ and the number of cycles of exact period $2$ is
		\[
		\frac{q^2-q}{2}=\frac{q(q-1)}{2}.
		\]
		Furthermore, a generic rational map of degree $q$ has a cycle of every prescribed exact period $m\ge 3$. 

		Suppose that $\bfm$ does not satisfy condition~\ref{cond_a}.  Then either all entries of $\bfm$ belong to $\{1,2\}$, or exactly one entry is greater than $2$ and all the remaining entries are equal to $2$.
		
		Consider first the latter case.  Removing two entries equal to $2$ leaves one entry greater than $2$ and $2d-5$ entries equal to $2$.  Since
		\[
		\frac{(d-1)(d-2)}{2}-(2d-5)
		=\frac{(d-3)(d-4)}{2}\ge 0
		\]
		for every $d\ge 3$, the resulting tuple is admissible for degree $d-1$.  Thus, possibility~\ref{cond_b} holds.
		
		It remains to consider the case in which all entries of $\bfm$ are equal to $1$ or $2$.  Let $r$ and $s$ denote the respective numbers of entries equal to $1$ and $2$, so that
		\[
		r+s=2d-2,
		\qquad r\le d.
		\]
		Assume first that $d\ge 5$.  If $r=d$, then $s=d-2$; remove one entry equal to $1$ and one entry equal to $2$.  The remaining tuple contains $d-1$ entries equal to $1$ and $d-3$ entries equal to $2$, and hence is admissible for degree $d-1$.  If $r\le d-1$, then $s\ge d-1\ge 2$, so we may remove two entries equal to $2$.  The number of remaining entries equal to $2$ is at most $2d-4$, and
		\[
		2d-4\le \frac{(d-1)(d-2)}{2}
		\]
		for $d\ge 5$.  The resulting tuple is therefore admissible for degree $d-1$.  Thus, possibility~\ref{cond_b} holds in both cases.
		
		For $d=3$, admissibility of $\bfm$ implies $r\ge 1$, since a generic cubic rational map has only three cycles of exact period $2$.  If $r=1$ or $r=2$, remove two entries equal to $2$; if $r=3$, remove one entry equal to $1$ and one entry equal to $2$.  In each case the remaining pair is admissible for degree $2$.  For $d=4$, if $r\in\{1,2,3\}$, remove two entries equal to $2$, while if $r=4$, remove one entry equal to $1$ and one entry equal to $2$.  The resulting tuple is admissible for degree $3$.  The only case not covered is $r=0$, namely
		\[
		\bfm=(2,2,2,2,2,2),
		\]
		which is precisely possibility~\ref{cond_c}.
	\end{proof}
	
	\begin{proof}[Proof of Theorem~\ref{main_theorem}]
		We proceed by induction on $d$.  For $d=2$, the assertion follows from Theorem~\ref{old_theorem}, which imposes no additional restriction on the periods in this degree.
		
		Now let $d\ge 3$ and assume that the theorem holds in degree $d-1$.  Let
		\[
		\bfm=(m_1,\ldots,m_{2d-2})
		\]
		be the tuple of periods of the selected periodic orbits.  By hypothesis, $\bfm$ is admissible.  We apply Lemma~\ref{lemma_admissible_reduction}.
		
		If possibility~\ref{cond_a} holds, then $\bfm$ satisfies both conditions of Theorem~\ref{old_theorem}: condition~\ref{cond1} follows from admissibility, while condition~\ref{cond2} is precisely possibility~\ref{cond_a}.  Hence the desired algebraic independence follows directly from Theorem~\ref{old_theorem}.
		
		Suppose next that possibility~\ref{cond_b} holds.  Remove the pair specified there and denote the resulting admissible tuple for degree $d-1$ by
		\[
		\bfm'=(m'_1,\ldots,m'_{2d-4}).
		\]
		By the induction hypothesis, the multiplier functions of periods $m'_1,\ldots,m'_{2d-4}$ are algebraically independent on $\Rat_{d-1}$.  Equivalently, the corresponding multiplier map on the marked-orbit space is dominant.  Since we work over $\bbC$, its differential has full rank on a nonempty Zariski open subset of the marked-orbit space.  Therefore, the projection of this subset to $\Rat_{d-1}$ is dense and constructible, and hence contains a nonempty Zariski open subset. Similarly, since the set $\mathcal V$ in Proposition~(\ref{prop_induction_step}) is Zariski open and nonempty, its projection to $\Rat_{d-1}$ is also a nonempty Zariski open subset. Since $\Rat_{d-1}$ is irreducible, these two open subsets intersect.  Thus, we may choose a map $f\in\Rat_{d-1}$ with non-multiple periodic orbits of respective periods $m'_1,\ldots,m'_{2d-4}$, whose local multiplier functions are locally independent at $f$, and such that $f$ belongs to the projection of $\mathcal V$. The fiber of $\mathcal V$ over $f$ is a nonempty Zariski open subset of $\bbC$, so a point $a_0\in\bbC$ may be chosen such that $(f,a_0)\in\mathcal V$ and $a_0$ does not belong to any of the selected orbits.

		Apply Proposition~\ref{prop_induction_step} in degree $d-1$.  If the pair removed from $\bfm$ was $(2,2)$, the proposition produces a degree $d$ rational map with locally independent multipliers of periods
		\[
		2,2,m'_1,\ldots,m'_{2d-4},
		\]
		which, up to reordering, is the tuple $\bfm$.  If the removed pair was $(1,2)$, the same proposition instead produces locally independent multipliers of periods
		\[
		1,2,m'_1,\ldots,m'_{2d-4},
		\]
		which again is $\bfm$ up to reordering.  In either case, Proposition~\ref{prop_local_to_global} implies that the multiplier functions corresponding to $\bfm$ are algebraically independent on $\Rat_d$.
		
		Finally, possibility~\ref{cond_c} is exactly the exceptional case $d=4$ in which all six periods are equal to $2$.  This case is supplied by Proposition~\ref{lemma_period_two_degree_four}.  The induction is complete.
	\end{proof}

	\section{Local multiplier asymptotics near a simple hole}\label{sec:added_pole}
	In this section we study the behavior of the period one and period two points that emerge near the pole $z=a$ of the rational map $g_{a,b}$ from~(\ref{g_ab_eq}) as $b\to 0$. The main purpose of this study is to give a proof of Proposition~\ref{prop_induction_step}.
	
	\begin{definition}\label{generic_def}
		We say that $a_0\in\bbC$ is a generic point for a degree $d\ge 1$ rational map $f\colon \widehat{\Cb}\to \widehat{\Cb}$ if the following conditions hold simultaneously:
		
		\begin{enumerate}[(1)]
			\item\label{cond_reg_1} $a_0$ is a regular value for $f$, and $\infty\not\in f^{-1}(a_0)$;
			\item $f(a_0)\neq 0,\infty$ and $f(a_0)-a_0\neq 0$;
			\item $f^2(a_0)\neq a_0$.
		\end{enumerate}
	\end{definition}
	
	In particular, condition~\ref{cond_reg_1} from Definition~\ref{generic_def} implies that if $a_0$ is a generic point for a degree $d\ge 1$ rational map $f$, then $a_0$ has $d$ distinct finite preimages $c_{1,0},\ldots, c_{d,0}$ under $f$ satisfying
	\[
	f(c_{i,0})=a_0,
	\qquad\text{and}\qquad
	f'(c_{i,0})\neq 0,
	\qquad \text{for all }i\in\{1,\ldots, d\}.
	\]
	Note that Definition~\ref{generic_def} is an open condition, so if $a_0\in\bbC$ is a generic point for a rational map $f$, then any point in a sufficiently small neighborhood $U$ of $a_0$ is also generic for $f$. Furthermore, if the neighborhood $U$ is simply connected, then for each $i\in\{1,\ldots, d\}$, the inverse branches
	\begin{equation}\label{inv_branch_eq}
		c_i=c_i(a)
	\end{equation}
	satisfying 
	\[
	f(c_i(a))=a,
	\qquad\text{and}\qquad
	c_i(a_0)=c_{i,0},
	\]
	are well defined and holomorphic in $U$.

	\begin{lemma}[Existence of the local fixed point and $d$ period-two orbits]\label{lem:existence-local-orbits}
		Assume that a point $a_0\in\bbC$ is generic for a degree $d\ge 1$ rational map $f$. Then there exist a neighborhood $U\subset\bbC$ of $a_0$, a neighborhood $V\subset\bbC$ of $0$, and $d+1$ distinct analytic maps 
		$$
		z_0,\ldots,z_d\colon U\times V\to \bbC
		$$ 
		such that for every $(a,b)\in U\times (V\setminus \{0\})$, the following holds. 
		\begin{enumerate}[(i)]
			\item\label{fixed_pt_part} $z_0(a,b)\in \bbC$ is a fixed point of the map $g_{a,b}$ from~(\ref{g_ab_eq}) satisfying the asymptotic identity
			\[
			z_0(a,b)
			=
			a+b\frac{f(a)}{a-f(a)}+O(b^2), 
			\]
			where all the $O(b^2)$-terms are meant as $b\to 0$ and uniformly with respect to $a\in U$;
			
			\item\label{inv_branch_part} the inverse branches $c_1(a),\ldots,c_d(a)$ from~(\ref{inv_branch_eq}) are distinct well-defined and holomorphic functions in $U$;
			
			\item\label{per_2_part} for each $i\in\{1,\ldots,d\}$, the point $z_i(a,b)\in \bbC$ is two-periodic for the map $g_{a,b}$ and satisfies 
			\[
			z_i(a,b)\to a
			\qquad\text{as }\quad b\to 0,
			\]
			with the second point of the two-cycle,
			\[
			w_i(a,b):=g_{a,b}(z_i(a,b))
			\]
			satisfying
			\[
			w_i(a,b)\to c_i(a)
			\qquad\text{as } \quad b\to 0.
			\]
			Furthermore, the following asymptotic identities hold:
			\[
			z_i(a,b)
			=
			a+b\frac{f(a)}{c_i(a)-f(a)}+O(b^2),
			\]
			and
			\[
			w_i(a,b)
			=
			c_i(a)
			+
			\frac{b}{f'(c_i(a))}
			\left(
			\frac{f(a)}{c_i(a)-f(a)}
			-
			\frac{a}{c_i(a)-a}
			\right)
			+O(b^2),
			\]
			where all the $O(b^2)$-terms are meant as $b\to 0$ and uniformly with respect to $a\in U$.
		\end{enumerate}
	\end{lemma}
	
	\begin{proof}
		Let $U\subset\bbC$ be a neighborhood of $a_0$ such that every point of $U$ is generic for $f$. Further conditions on $U$ will be stated later in the proof.

		We first give a proof of part~\ref{fixed_pt_part}. For fixed $a, b$, make the substitution
		\[
		z=a+bu.
		\]
		Then the fixed point equation $g_{a,b}(z)=z$ is equivalent to
		$$
		a+bu
		=
		f(a+bu)\left(1+\frac1u\right),
		$$
		which can be rewritten as 
		\begin{equation}\label{F_impl_eq}
			F(a,b,u) = 0,\qquad\text{where}\qquad F(a,b,u)
			:=
			a+bu
			-
			f(a+bu)\left(1+\frac1u\right).	
		\end{equation}
		For $b=0$, the equation~(\ref{F_impl_eq}) becomes
		\[
		a=f(a)\left(1+\frac1u\right),
		\]
		and hence has the solution
		\[
		u_0(a):=\frac{f(a)}{a-f(a)}.
		\]
		Since $a\in U$ is a generic point for $f$, it follows that $u_0(a)\neq 0,\infty$. Moreover,
		\[
		\frac{\partial F}{\partial u}(a,0,u_0(a))
		=
		\frac{f(a)}{u_0(a)^2},
		\]
		is well defined and is a finite and nonzero number for all $a\in U$, again due to the fact that $a$ is a generic point for $f$. By the holomorphic implicit function theorem, after shrinking $U$ and choosing a sufficiently small neighborhood $V$ of $0$, there exists a holomorphic function
		\[
		u_0(a,b)
		\]
		on $U\times V$, satisfying
		\[
		u_0(a,0)=\frac{f(a)}{a-f(a)}
		\]
		and
		\[
		F(a,b,u_0(a,b))=0.
		\]
		Therefore
		\[
		z_0(a,b):=a+b u_0(a,b)
		\]
		is a fixed point of $g_{a,b}$, for $b\in V\setminus\{0\}$. Since $u_0(a,b)$ is holomorphic on $U\times V$, after shrinking the neighborhood $U$ again, we have
		\[
		u_0(a,b)=\frac{f(a)}{a-f(a)}+O(b),
		\]
		uniformly for $a\in U$. Hence
		\[
		z_0(a,b)
		=
		a+b\frac{f(a)}{a-f(a)}+O(b^2),
		\]
		uniformly for $a\in U$, which completes the proof of part~\ref{fixed_pt_part}.
		
		Next, we shrink the neighborhood $U$ again if necessary, so that $U$ is bounded and simply connected. Since $U$ consists of generic points of $f$, this immediately implies part~\ref{inv_branch_part}.
		
		We now give a proof of part~\ref{per_2_part}. Fix $i\in\{1,\ldots,d\}$. Again make the substitution
		\[
		z=a+bu,
		\]
		and denote the second point of the two-cycle by $w$. The equations
		\[
		g_{a,b}(z)=w,
		\qquad
		g_{a,b}(w)=z
		\]
		are equivalent to
		\[
		w
		=
		f(a+bu)\left(1+\frac1u\right),\qquad\text{and}\qquad 
		a+bu
		=
		f(w)\left(1+\frac{b}{w-a}\right)
		\]
		which can be rewritten as 
		\begin{equation}\label{equation_per_2}
			\Phi_1(a,b,u,w) = 0\qquad\text{and}\qquad \Phi_2(a,b,u,w) = 0,
		\end{equation}
		where
		\[
		\Phi_1(a,b,u,w)
		:=
		w
		-
		f(a+bu)\left(1+\frac1u\right),
		\]
		and
		\[
		\Phi_2(a,b,u,w)
		:=
		a+bu
		-
		f(w)\left(1+\frac{b}{w-a}\right).
		\]
		At $b=0$, the equations~(\ref{equation_per_2}) become
		\[
		w=f(a)\left(1+\frac1u\right),
		\qquad
		a=f(w).
		\]
		Choosing $w=c_i(a)$, we get the solution
		\[
		u(a) = u_i(a):=\frac{f(a)}{c_i(a)-f(a)},
		\qquad
		w(a) = w_i(a)=c_i(a).
		\]
		Since $a\in U$ is a generic point for $f$, it follows that $u_i(a)\neq 0,\infty$ and $w=c_i(a)\neq a,\infty$. Hence, the Jacobian matrix of $(\Phi_1,\Phi_2)$ with respect to $(u,w)$, evaluated at
		\[
		(a,0,u_i(a),c_i(a)),
		\]
		is well defined and is equal to 
		\[
		\begin{pmatrix}
			\dfrac{f(a)}{u_i(a)^2} & 1\\[1.2em]
			0 & -f'(c_i(a))
		\end{pmatrix}.
		\]
		Its determinant equals
		\[
		-\frac{f(a)f'(c_i(a))}{u_i(a)^2},
		\]
		which is nonzero on $U$. Hence the holomorphic implicit function theorem gives, after possibly shrinking $U$ and $V$, holomorphic functions
		\[
		u_i(a,b),\qquad w_i(a,b),
		\]
		defined on $U\times V$, such that
		\[
		u_i(a,0)=\frac{f(a)}{c_i(a)-f(a)},
		\qquad
		w_i(a,0)=c_i(a),
		\]
		and
		\[
		\Phi_1(a,b,u_i(a,b),w_i(a,b))=0,
		\qquad
		\Phi_2(a,b,u_i(a,b),w_i(a,b))=0.
		\]
		Thus, for $b\in V\setminus\{0\}$,
		\[
		z_i(a,b):=a+b u_i(a,b)
		\]
		is a period 2 point for the map $g_{a,b}$ and satisfies
		\[
		g_{a,b}(z_i(a,b))=w_i(a,b),
		\qquad
		g_{a,b}(w_i(a,b))=z_i(a,b).
		\]
		
		Since $u_i(a,b)$ is holomorphic on $U\times V$, we have
		\[
		u_i(a,b)
		=
		\frac{f(a)}{c_i(a)-f(a)}
		+
		O(b),
		\]
		uniformly for $a\in U$, after possibly shrinking the neighborhood $U$ again. Therefore
		\[
		z_i(a,b)
		=
		a+b\frac{f(a)}{c_i(a)-f(a)}+O(b^2),
		\]
		uniformly for $a\in U$. In particular,
		\[
		z_i(a,b)\to a
		\qquad\text{as } b\to 0.
		\]
		Also,
		\[
		w_i(a,b)\to c_i(a)
		\qquad\text{as } b\to 0,
		\]
		which implies that all functions $z_i$ are distinct.
		
		It remains to compute the first-order expansion of $w_i(a,b)$. Write
		\[
		w_i(a,b)=c_i(a)+b v_i(a)+O(b^2),
		\]
		for some function $v_i(a)$ that needs to be determined.
		Using
		\[
		z_i(a,b)
		=
		a+b u_i(a)+O(b^2),
		\qquad
		u_i(a)=\frac{f(a)}{c_i(a)-f(a)},
		\]
		and expanding the equation
		\[
		z_i(a,b) = g_{a,b}(w_i(a,b)),
		\]
		we get
		\[
		a+b u_i(a)+O(b^2)
		=
		\left(a+b f'(c_i(a))v_i(a)+O(b^2)\right)
		\left(
		1+\frac{b}{c_i(a)-a}+O(b^2)
		\right).
		\]
		Comparing the coefficients of $b$, we obtain
		\[
		u_i(a)
		=
		f'(c_i(a))v_i(a)
		+
		\frac{a}{c_i(a)-a}.
		\]
		Therefore
		\[
		v_i(a)
		=
		\frac{1}{f'(c_i(a))}
		\left(
		\frac{f(a)}{c_i(a)-f(a)}
		-
		\frac{a}{c_i(a)-a}
		\right).
		\]
		Hence
		\[
		w_i(a,b)
		=
		c_i(a)
		+
		\frac{b}{f'(c_i(a))}
		\left(
		\frac{f(a)}{c_i(a)-f(a)}
		-
		\frac{a}{c_i(a)-a}
		\right)
		+
		O(b^2),
		\]
		uniformly for $a\in U$.
		
		Finally, since $c_i(a)\neq a$ on $U$, we have
		\[
		z_i(a,b)\to a,
		\qquad
		w_i(a,b)\to c_i(a)\neq a.
		\]
		Thus, for sufficiently small $V$, the two points $z_i(a,b)$ and $w_i(a,b)$ are distinct for every $a\in U$ and every $b\in V\setminus\{0\}$. Therefore $z_i(a,b)$ has exact period $2$.
	\end{proof}
	
	In the following lemmas we give further asymptotic estimates related to the multipliers of the fixed point $z_0(a,b)$ and the period 2 points $z_1(a,b),\ldots,z_d(a,b)$ of the maps $g_{a,b}$. All computations are valid for $(a,b)$ in the neighborhood $U\times (V\setminus\{0\})$ from Lemma~\ref{lem:existence-local-orbits}. In all cases, for any $k\in\bbZ$, the $O(b^k)$-terms are meant as $b\to 0$ and are uniform with respect to $a\in U$.

	\begin{lemma}[Multiplier of the local fixed point]\label{lem:fixed-multiplier}
		Let
		\[
		\nu(a,b):=g_{a,b}'(z_0(a,b))
		\]
		be the multiplier of the fixed point $z_0(a,b)$ from Lemma~\ref{lem:existence-local-orbits}. Then, as $b\to 0$,
		\[
		\nu(a,b)=\frac{D(a)}{b}+O(1),
		\]
		where
		\begin{equation}\label{Da_eq}
			D(a)=-\frac{(a-f(a))^2}{f(a)}.
		\end{equation}
		Furthermore,
		\[
		\frac{\partial \nu}{\partial b}(a,b)
		=
		-\frac{D(a)}{b^2}+O(1),
		\]
		and
		\[
		\frac{\partial \nu}{\partial a}(a,b)
		=
		\frac{D'(a)}{b}+O(1).
		\]
	\end{lemma}
	
	\begin{proof}
		Differentiating~(\ref{g_ab_eq}), we get
		\begin{equation}\label{g_prime_eq}
			g_{a,b}'(z)
			=
			f'(z)\left(1+\frac{b}{z-a}\right)
			-
			f(z)\frac{b}{(z-a)^2},
		\end{equation}
		while Lemma~\ref{lem:existence-local-orbits} provides the asymptotic estimate
		\[
		z_0(a,b)-a
		=
		b\frac{f(a)}{a-f(a)}+O(b^2).
		\]
		Substituting it into~(\ref{g_prime_eq}) and using 
		\[
		f(z_0)=f(a)+O(b)\qquad\text{and}\qquad f'(z_0)=O(1),
		\]
		we get
		\[
		\nu(a,b)=\frac{D(a)}{b}+O(1).
		\]
		The implicit-function construction above shows that the difference $\nu(a,b)-D(a)/b$ extends holomorphically across $b=0$, and hence its partial derivatives with respect to $a$ and $b$ are $O(1)$. 
		Hence differentiating the Laurent expansion gives the asserted estimates for $\partial_b\nu$ and $\partial_a\nu$.
	\end{proof}

	\begin{lemma}[Multiplier of a local period-two orbit]\label{lem:period-two-multiplier}
		For each $i\in\{1,\ldots,d\}$, let
		\[
		\mu_i(a,b):=g_{a,b}'(z_i(a,b))\,g_{a,b}'(w_i(a,b))
		\]
		be the multiplier of the period-two point $z_i(a,b)$ from Lemma~\ref{lem:existence-local-orbits}. Then, as $b\to 0$,
		\[
		\mu_i(a,b)=\frac{C_i(a)}{b}+O(1),
		\]
		where
		\begin{equation}\label{Cia_eq}
			C_i(a)=
			-\frac{f'(c_i(a))\bigl(c_i(a)-f(a)\bigr)^2}{f(a)}.
		\end{equation}
		Furthermore,
		\[
		\frac{\partial \mu_i}{\partial b}(a,b)
		=
		-\frac{C_i(a)}{b^2}+O(1),
		\]
		and
		\[
		\frac{\partial \mu_i}{\partial a}(a,b)
		=
		\frac{C_i'(a)}{b}+O(1).
		\]
	\end{lemma}
	
	\begin{proof}
		By Lemma~\ref{lem:existence-local-orbits},
		\[
		z_i(a,b)-a
		=
		b\frac{f(a)}{c_i(a)-f(a)}+O(b^2).
		\]
		Substituting this into~(\ref{g_prime_eq}), we get
		\[
		g_{a,b}'(z_i)
		=
		-\frac{(c_i(a)-f(a))^2}{b f(a)}+O(1).
		\]
		On the other hand, since $w_i(a,b)\to c_i(a)$ and $c_i(a)\neq a$,
		\[
		g_{a,b}'(w_i)=f'(c_i(a))+O(b).
		\]
		Multiplying the two estimates gives
		\[
		\mu_i(a,b)
		=
		-\frac{f'(c_i(a))(c_i(a)-f(a))^2}{b f(a)}+O(1)
		=
		\frac{C_i(a)}{b}+O(1).
		\]
		The implicit-function construction above shows that the difference $\mu_i(a,b)-C_i(a)/b$ extends holomorphically across $b=0$, and hence its partial derivatives with respect to $a$ and $b$ are $O(1)$. 
		Hence, the asymptotic formulas for $\partial_a\mu_i$ and $\partial_b\mu_i$ follow.	
	\end{proof}

	For any pair $i, j\in\{1,\ldots,d\}$ with $i\neq j$ and any $(a,b)\in U\times (V\setminus\{0\})$, we consider the determinant of the $2\times 2$ Jacobian matrix
	\begin{equation}\label{J_ij_eq}
	J_{i,j}(a,b):=  \det
	\begin{pmatrix}
		\partial_a\mu_i(a,b) & \partial_b\mu_i(a,b)\\
		\partial_a\mu_j(a,b) & \partial_b\mu_j(a,b)
	\end{pmatrix}.		
	\end{equation}
	Similarly, for any $i\in \{1,\ldots,d\}$ and any $(a,b)\in U\times (V\setminus\{0\})$, we consider the determinant of the $2\times 2$ Jacobian matrix
	$$
	J_{0,i}(a,b)
	:=
	\det
	\begin{pmatrix}
		\partial_a\nu(a,b) & \partial_b\nu(a,b)\\
		\partial_a\mu_i(a,b) & \partial_b\mu_i(a,b)
	\end{pmatrix}.
	$$
	Applying Lemma~\ref{lem:fixed-multiplier} and Lemma~\ref{lem:period-two-multiplier}, for any $i, j\in\{1,\ldots,d\}$ with $i\neq j$, we obtain the asymptotic formulas
	\[
	J_{0,i}(a,b)
	=
	\frac{D(a)C_i'(a)-D'(a)C_i(a)}{b^3}
	+O\left(\frac1{b^2}\right)
	\]
	and
	\[
	J_{i,j}(a,b)
	=
	\frac{C_i(a)C_j'(a)-C_i'(a)C_j(a)}{b^3}
	+O\left(\frac1{b^2}\right),
	\]
	where $D(a)$ and $C_i(a)$ are defined in~(\ref{Da_eq}) and~(\ref{Cia_eq}) respectively. We define the leading terms of these expressions:
	\begin{equation}\label{Phi_eq}
		\Phi_i(a):= D(a)C_i'(a)-D'(a)C_i(a),
	\end{equation}
	\begin{equation}\label{Psi_eq}
		\Psi_{i,j}(a):= C_i(a)C_j'(a)-C_i'(a)C_j(a).
	\end{equation}

	\begin{lemma}\label{lem:generic-jacobian-nondegeneracy}
		Assume $d\ge 2$. Then there exists a nonempty Zariski open subset $\mathcal V\subset\Rat_d\times\bbC$, such that for any $(f,a)\in\mathcal V$, the point $a$ is generic for $f$ and satisfies $\Phi_i(a)\neq 0$ and $\Psi_{i,j}(a)\neq 0$, for any $i,j\in\{1,\ldots,d\}$ with $i\neq j$.	
	\end{lemma}
	
	\begin{proof}
		First, we observe that to prove the lemma, it is enough to exhibit a pair $(f,a)$ for which all the functions $\Phi_i(a)$ and $\Psi_{i,j}(a)$ are defined and nonzero. Indeed, consider the Zariski open locus $B\subset\Rat_d\times\bbC$ that consists of all pairs $(f,a)$ where $a$ is generic for $f$. Consider the incidence varieties that consist of all elements of $B$ with one or two distinct marked preimages of $a$, given respectively by
		\[
		f(c_i)=a
		\qquad\text{and}\qquad
		f(c_i)=f(c_j)=a,\quad c_i\neq c_j.
		\]
		Their projections to $B$ are finite. Moreover, if $c=c(a)$ is a local inverse branch of $f$, then
		\[
		c'(a)=\frac{1}{f'(c(a))}.
		\]
		Thus differentiation with respect to $a$ along an inverse branch is a rational derivation on the corresponding incidence variety. Since, according to~(\ref{Da_eq}) and~(\ref{Cia_eq}), $D$ and the $C_i$ are rational functions there, it follows that $\Phi_i$ and $\Psi_{i,j}$ are rational functions on the respective incidence varieties.
		
		The loci \(G_{\Phi_i}\) and \(G_{\Psi_{i,j}}\) where these rational functions are defined and nonzero are Zariski open. 		
		Since the incidence projections are finite, the images of the complements of \(G_{\Phi_i}\) and \(G_{\Psi_{i,j}}\) are closed in $B$. Therefore, to prove the lemma, it is enough to exhibit a pair $(f,a)\in B$ such that all the functions $\Phi_i(a)$ and $\Psi_{i,j}(a)$ are defined and nonzero. 
		
		We use the polynomial map
		\[
		f(z)=z^d.
		\]
		Write
		$$
		a=t^d,\qquad \text{for some}\quad t\in\bbC
		$$
		and set
		\begin{equation}\label{c_zeta_eq}
			c_\zeta=\zeta t,
			\qquad \text{where }\zeta \text{ is a root of unity}\qquad
			\zeta^d=1.
		\end{equation}
		Then $f(c_\zeta)=a$, and we obtain 
		\[
		f(a)=t^{d^2},
		\qquad
		f'(c_\zeta)=d\zeta^{-1}t^{d-1}
		\]
		and
		\[
		D(a)
		=
		-t^{2d-d^2}\left(1-t^{d^2-d}\right)^2.
		\]
		In what follows, we are going to index the functions $C_i$, $\Phi_i$ and $\Psi_{i,j}$ by the roots of unity in~(\ref{c_zeta_eq}) rather than by the numbers from $1$ to $d$. We obtain
		\[
		C_\zeta(a)
		=
		-d\zeta^{-1}t^{d+1-d^2}
		\left(\zeta-t^{d^2-1}\right)^2
		\]
		and
		\begin{equation}\label{C_over_D_power_map_eq}
			\frac{C_\zeta(a)}{D(a)}
			=
			d\zeta^{-1}t^{1-d}
			\frac{\left(\zeta-t^{d^2-1}\right)^2}
			{\left(1-t^{d^2-d}\right)^2}.
		\end{equation}
		Since $d\ge 2$, the right-hand side has a pole at $t=0$ and is therefore not constant. Hence
		\[
		\left(\frac{C_\zeta}{D}\right)'\not\equiv 0.
		\]
		Since
		\[
		\Phi_\zeta
		=
		D^2\left(\frac{C_\zeta}{D}\right)',
		\]
		we obtain $\Phi_\zeta\not\equiv 0$ for every $d$-th root of unity $\zeta$.
		
		Similarly, if $\zeta\neq\eta$ are two distinct $d$-th roots of unity, then
		\begin{equation}\label{C_ratio_power_map_eq}
			\frac{C_\zeta(a)}{C_\eta(a)}
			=
			\frac{\eta}{\zeta}
			\left(
			\frac{\zeta-t^{d^2-1}}{\eta-t^{d^2-1}}
			\right)^2.
		\end{equation}
		This function is not constant, because $\zeta\neq\eta$. Hence
		\[
		\left(\frac{C_\zeta}{C_\eta}\right)'\not\equiv 0.
		\]
		Since
		\[
		\Psi_{\zeta,\eta}
		=
		-C_\eta^2\left(\frac{C_\zeta}{C_\eta}\right)',
		\]
		we obtain $\Psi_{\zeta,\eta}\not\equiv 0$ for every pair $\zeta\neq\eta$.
		
		Thus, for the map $f(z)=z^d$, none of the functions $\Phi_i$ and $\Psi_{i,j}$ is identically zero. Since there are only finitely many of them, we can choose a generic value of $a$ at which they are all defined and nonzero. The finite-projection argument above now gives a nonempty Zariski open subset $\mathcal V\subset\Rat_d\times\bbC$ such that any $(f,a)\in\mathcal V$ has the required properties. This completes the proof of Lemma~\ref{lem:generic-jacobian-nondegeneracy}.
	\end{proof}

	\begin{proof}[Proof of Proposition~\ref{prop_induction_step}]
		
		We show that the set $\mathcal V\subset\Rat_d\times\bbC$ can be chosen the same as in Lemma~\ref{lem:generic-jacobian-nondegeneracy}. Let $(f, a_0)\in\mathcal V$ be an arbitrary point and let
		\[
		\mathcal O_1,\ldots,\mathcal O_{2d-2}
		\]
		be non-multiple periodic orbits of $f$ of respective periods
		\[
		m_1,\ldots,m_{2d-2}
		\]
		with the multipliers being locally independent at $f$ 
		and with 
		\[
		a_0\notin \mathcal O_k
		\qquad\text{for all } k=1,\ldots,2d-2.
		\]
		Since the multipliers of the orbits $\mathcal O_1,\ldots,\mathcal O_{2d-2}$ are locally independent at $f$, there exists a holomorphic family
		\[
		f_{\tau}\in\Rat_d,
		\qquad
		\text{with}\qquad\tau=(\tau_1,\ldots,\tau_{2d-2})
		\]
		defined for $\tau$ in a neighborhood of $0\in\bbC^{2d-2}$, such that $f_0=f$ and
		\begin{equation}\label{old_multiplier_jacobian_eq}
			\det\left(
			\frac{\partial \lambda_j(f_\tau)}{\partial \tau_k}\bigg|_{\tau=0}
			\right)_{j,k=1}^{2d-2}
			\neq 0,
		\end{equation}
		where each $\lambda_j(f_\tau)$ denotes the local multiplier function corresponding to the analytic continuation of the respective orbit~$\mathcal O_j$.
		
		For the parameters $a$ near $a_0$ and $b$ near $0$, set
		\[
		G_{\tau,a,b}(z)
		:=
		f_\tau(z)\frac{z-a+b}{z-a}.
		\]
		Note that for fixed parameters $a$ and $\tau$, the maps $G_{\tau,a,b}$ converge to $f_\tau$ uniformly on compact subsets of $\widehat\bbC\setminus \{a\}$ when $b\to 0$. Since the orbits $\mathcal O_j$ of $f_0$ avoid $a_0$ and have multipliers different from $1$, after shrinking the parameter neighborhood if necessary, these orbits continue holomorphically as periodic orbits of $G_{\tau,a,b}$. For each $j\in\{1,\ldots,2d-2\}$, let
		\[
		\Lambda_j(\tau,a,b)
		\]
		be the corresponding multipliers of these periodic orbits viewed as periodic orbits of $G_{\tau,a,b}$. It follows that they are analytic as functions of $(\tau, a, b)$ in a neighborhood of $(0, a_0, 0)$. Hence
		\begin{equation}\label{old_multiplier_estimates_eq}
			\frac{\partial \Lambda_j}{\partial \tau_k}(0,a_0,b)
			=
			\frac{\partial \lambda_j(f_\tau)}{\partial \tau_k}\bigg|_{\tau=0}
			+O(b),
		\end{equation}
		\begin{equation}\label{dLambda_ab_eq}
		\frac{\partial \Lambda_j}{\partial a}(0,a_0,b)=O(b),
		\qquad
		\frac{\partial \Lambda_j}{\partial b}(0,a_0,b)=O(1),
		\end{equation}
		as $b\to 0$. Here, the terms $O(b)$ and $O(1)$ are uniform in $a$ and $\tau$ when $a-a_0$ and $\tau$ are sufficiently small.
		
		We first add two period-two multipliers. Choose two distinct period $2$ points of $G_{\tau,a,b}$, say $z_1$ and $z_2$, constructed in Lemma~\ref{lem:existence-local-orbits}. Let
		\begin{equation}\label{mu1mu2_eq}
			\mu_1(\tau,a,b),
			\qquad\text{and}\qquad
			\mu_2(\tau,a,b)		
		\end{equation}
		be their multipliers.
		
		By Lemma~\ref{lem:period-two-multiplier}, applied uniformly to the nearby maps $f_\tau$, we have
		\[
		\mu_i(\tau,a,b)
		=
		\frac{C_i(\tau,
			a)}{b}+O(1),
		\qquad i=1,2,
		\]
		where $C_i(\tau,a_0)$ is defined by the same expression as $C_i(a_0)$ from~(\ref{Cia_eq}), with $f$ substituted by $f_\tau$. Therefore, we have
		\begin{equation}\label{dMu_tau_eq}
		\frac{\partial \mu_i}{\partial \tau_k}(0,a_0,b)=O\left(\frac1b\right).
		\end{equation}
		Consider the $2d\times 2d$ Jacobian matrix $J(b)$ of the multipliers
		\[
		\Lambda_1,\ldots,\Lambda_{2d-2},\mu_1,\mu_2
		\]
		with respect to the parameters
		\[
		\tau_1,\ldots,\tau_{2d-2},a,b,
		\]
		evaluated at $\tau_1=\ldots=\tau_{2d-2}=0$, $a=a_0$ and $b$ in a punctured neighborhood of zero. That is,
		$$
		J(b) := \left(\left.
		\frac{\partial(\Lambda_1,\ldots,\Lambda_{2d-2},\mu_1,\mu_2)}
		{\partial(\tau_1,\ldots,\tau_{2d-2},a,b)}
		\right|_{\tau=0, a=a_0}\right).
		$$
		Write $J(b)$ in the block form as
		$$
		J(b) = 
			\begin{pmatrix}
			A(b) & \vline &X(b) \\
			\hline
			Y(b) &  \vline &M(b)	
		\end{pmatrix},
		$$
		where $A(b)$ is the $(2d-2)\times (2d-2)$ block formed by differentiating the old multipliers $\Lambda_j$ with respect to the parameters $\tau_k$, and $M(b)$ is the $2\times 2$ matrix obtained by differentiating $\mu_1$ and $\mu_2$ with respect to $a$ and $b$.
		
		It follows from~(\ref{old_multiplier_estimates_eq}),~(\ref{dLambda_ab_eq}),~(\ref{dMu_tau_eq}) and Lemma~\ref{lem:period-two-multiplier} that
		\[
		A(b)=
		\left(
		\frac{\partial \lambda_j(f_\tau)}{\partial \tau_k}\bigg|_{\tau=0}
		\right)_{j,k=1}^{2d-2}
		+O(b),
		\qquad Y(b) = O\left(b^{-1}\right),
		\]
		and for the last two columns of $J(b)$, there is an asymptotic relation
		$$
		\begin{pmatrix}
			X(b) \\
			\hline
			M(b)	
		\end{pmatrix}
		=
		\begin{pmatrix}
			O(b) & O(1) \\
			\hline
			O(b^{-1}) &  O(b^{-2})	
		\end{pmatrix}
		$$
		Now, after dividing the second to last column of $J(b)$ by $b$ and multiplying the last two rows of $J(b)$ by $b^2$, it becomes evident that 		
		the only contribution of order $b^{-3}$ to the determinant $\det J(b)$ comes from the product of the old Jacobian block $A(b)$ and the $2\times 2$ Jacobian block $M(b)$. The computation, preceding Lemma~\ref{lem:generic-jacobian-nondegeneracy}, also yields
		$$
		\det M(b) = \frac{\Psi_{1,2}(a_0)}{b^3}
		+
		O\left(\frac1{b^2}\right),
		$$
		so we get
		\[
		\det J(b)
		=
		\det\left(
		\frac{\partial \lambda_j(f_\tau)}{\partial \tau_k}\bigg|_{\tau=0}
		\right)
		\frac{\Psi_{1,2}(a_0)}{b^3}
		+
		O\left(\frac1{b^2}\right).
		\]
		Since $(f,a_0)\in \mathcal V$, Lemma~\ref{lem:generic-jacobian-nondegeneracy} gives
		\[
		\Psi_{1,2}(a_0)\neq 0.
		\]
		Thus, together with~(\ref{old_multiplier_jacobian_eq}), this implies that the determinant $\det J(b)$ is nonzero for all sufficiently small nonzero $b$. Hence, for such $b$, the map $g_{a_0,b}=G_{0,a_0,b}$ has locally independent multipliers of periods
		\[
		2,2,m_1,
		\ldots,m_{2d-2}.
		\]

		The argument for periods
		\[
		1,2,m_1,
		\ldots,m_{2d-2}
		\]
		is the same. Let $\nu(\tau,a,b)$ be the multiplier of the local fixed point of $G_{\tau,a,b}$ from Lemma\ref{lem:fixed-multiplier}, and let $\mu_1(\tau,a,b)$ be the multiplier of the local period-two orbit from~(\ref{mu1mu2_eq}). By Lemmas~\ref{lem:fixed-multiplier} and~\ref{lem:period-two-multiplier},
		\[
		\det
		\begin{pmatrix}
			\partial_a \nu(0,a_0,b) & \partial_b \nu(0,a_0,b)\\
			\partial_a \mu_1(0,a_0,b) & \partial_b \mu_1(0,a_0,b)
		\end{pmatrix}
		=
		\frac{\Phi_1(a_0)}{b^3}+O\left(\frac1{b^2}\right).
		\]
		Therefore
		\[
		\det\left(\left.
		\frac{\partial(\Lambda_1,\ldots,\Lambda_{2d-2},\nu,\mu_1)}
		{\partial(\tau_1,\ldots,\tau_{2d-2},a,b)}
		\right|_{\tau=0,a=a_0}\right)
		=
		\det\left(
		\frac{\partial \lambda_j(f_\tau)}{\partial \tau_k}\bigg|_{\tau=0}
		\right)
		\frac{\Phi_1(a_0)}{b^3}
		+
		O\left(\frac1{b^2}\right).
		\]
		Since $(f,a_0)\in \mathcal V$, Lemma~\ref{lem:generic-jacobian-nondegeneracy} gives
		\[
		\Phi_1(a_0)\neq 0.
		\]
		Thus this determinant is also nonzero for all sufficiently small nonzero $b$.
		
		Choosing $b\neq 0$ sufficiently close to $0$ so that both determinants are nonzero proves the proposition.
	\end{proof}
	
	\section{A three-hole degeneration and the period two multipliers in degree $4$}\label{sec:lemma_proof}
	
	It remains to give a proof of Proposition~\ref{lemma_period_two_degree_four}. The idea of the proof is to degenerate a degree $4$ rational map to a M\"obius map with three holes. All six period-two cycles that we use are created by the degeneration. We first carry out the construction for a generic M\"obius map and generic hole positions and reduce the computation of the six-by-six Jacobian to the computation of the determinant of a three-by-three submatrix. Then we choose a particular three-parameter family for which the latter determinant is easy to evaluate.
	
	Let $M\colon\bbP\to\bbP$ be a M\"obius transformation and let $a_1,a_2,a_3\in\mathbb C$ be three distinct points. For a vector
	\[
	\mathbf b=(b_1,b_2,b_3)\in\bbC^3
	\]
	consider the rational map
	\begin{equation}\label{eq:Ggeneral}
		G_{M,\mathbf a,\mathbf b}(z)
		=M(z)\prod_{r=1}^3\left(1+\frac{b_r}{z-a_r}\right),
	\end{equation}
	indexed by the M\"obius map $M$, the vector of holes $\mathbf a=(a_1,a_2,a_3)$ and the vector $\bfb$. For generic parameters with $b_1b_2b_3\neq0$, this is a rational map of degree $4$, whereas for $\mathbf b=0$ it degenerates to $M$.
	
	For each $j\in\{1,2,3\}$ define
	$$
	c_j:= M^{-1}(a_j).
	$$
	
	\begin{lemma}[Period-two cycles near three holes]\label{lem:three-hole-period-two-cycles}
		Assume that the following conditions hold:
		\begin{enumerate}[(i)]
			\item\label{cond_1} the points $a_1,a_2,a_3$ are distinct;
			\item\label{cond_2} for every $r\in\{1,2,3\}$, the point $c_r=M^{-1}(a_r)$ is finite and does not belong to $\{a_1,a_2,a_3\}$;
			\item\label{cond_3} for every $r\in\{1,2,3\}$, the value $M(a_r)$ is finite and nonzero, and
			\[
			M(a_r)\notin\{c_r\}\cup\{a_s:s\neq r\}.
			\]
		\end{enumerate}
		Then, for all sufficiently small $b_1,b_2,b_3\in\bbC^*$, the map $G_{M,\mathbf a,\mathbf b}$ from~(\ref{eq:Ggeneral}) has six distinct cycles of exact period $2$. Three of them, indexed by $r\in\{1,2,3\}$, have points converging to $a_r$ and $c_r$, respectively, as $\mathbf b\to0$. If $\mu_r$ denotes the multiplier of this cycle, then
		\[
		\mu_r
		=
		\frac{A_r(M,\mathbf a)}{b_r}
		\left(1+O(\|\mathbf b\|)\right),
		\]
		where
		\[
		A_r(M,\mathbf a)
		:=
		-\frac{M'(c_r)\bigl(c_r-M(a_r)\bigr)^2}{M(a_r)}.
		\]
		The other three cycles are indexed by the pairs $1\le r<s\le3$ and have points converging to $a_r$ and $a_s$, respectively. If $\mu_{rs}$ denotes the multiplier of the cycle corresponding to the pair $(r,s)$, then
		\[
		\mu_{rs}
		=
		\frac{A_{rs}(M,\mathbf a)}{b_rb_s}
		\left(1+O(\|\mathbf b\|)\right),
		\]
		where
		\[
		A_{rs}(M,\mathbf a)
		:=
		\frac{\bigl(a_s-M(a_r)\bigr)^2
			\bigl(a_r-M(a_s)\bigr)^2}
		{M(a_r)M(a_s)}.
		\]
		All the asymptotic estimates are locally uniform in $(M,\mathbf a)$ as long as conditions \textup{(i)}--\textup{(iii)} remain satisfied.
	\end{lemma}
	
	\begin{proof}
		For the first three periodic orbits, indexed by a single number $r\in\{1,2,3\}$, the statement of the lemma follows directly from Lemma~\ref{lem:existence-local-orbits} and Lemma~\ref{lem:period-two-multiplier}. Indeed, if we take 
		$$
		f(z) = G_{M,\mathbf a,\mathbf b}(z) \left(1+\frac{b_r}{z-a_r}\right)^{-1},
		$$
		then existence of the corresponding $2$-cycle follows from Lemma~\ref{lem:existence-local-orbits}, while Lemma~\ref{lem:period-two-multiplier} implies that 
		$$
		\mu_r = \frac{A_r(f,\mathbf a)}{b_r} + O(1),
		$$
		where
		\[
		A_r(f,\mathbf a):=-\frac{f'(c_r(f))\bigl(c_r(f)-f(a_r)\bigr)^2}{f(a_r)}
		\]
		and $c_r(f)$ denotes the inverse branch satisfying $f(c_r(f))=a_r$ and converging to $c_r$ as $\bfb\to0$. 		The application of Lemmas~\ref{lem:existence-local-orbits} and~\ref{lem:period-two-multiplier} is justified because, for all sufficiently small $b_j\neq 0$ with $j\neq r$, the point $a_r$ is generic for this auxiliary map $f$. Indeed, the three preimages of $a_r$ under $f$ are finite and simple and converge, respectively, to $c_r$ and to the two points $a_j$ with $j\neq r$, while conditions~\ref{cond_2} and~\ref{cond_3} also imply, for sufficiently small parameters, that $f(a_r)\notin\{0,\infty,a_r\}$ and $f^2(a_r)\neq a_r$.

		Finally, as $\bfb\to 0$, the map $f$ converges to $M$ uniformly on compact subsets not containing the holes. Since according to condition~\ref{cond_2}, the point $c_r$ stays away from the holes, the required asymptotic estimate  
		\[
		\mu_r
		=
		\frac{A_r(M,\mathbf a)}{b_r}
		\left(1+O(\|\mathbf b\|)\right)
		\]
		follows.
		
		For the remaining $2$-cycles, the proof follows the same strategy as the proofs of Lemma~\ref{lem:existence-local-orbits} and Lemma~\ref{lem:period-two-multiplier}.
		
		Fix the parameters $\bfa,\bfb$ and the M\"obius map $M$. For a selected pair of indices $r$, $s$, satisfying $1\le r<s\le 3$, write the points of the prospective cycle near $a_r$ and $a_s$ as
		\[
		z=a_r+b_ru,
		\qquad w=G_{M,\mathbf a,\mathbf b}(z)=a_s+b_sv.
		\]
		Then, the periodicity condition can be rewritten in the form
		$$
		\Phi_1(\bfa,\bfb,M, u, v)=0 \qquad\text{and}\qquad \Phi_2(\bfa,\bfb,M, u, v)=0,
		$$
		where
		$$
		\Phi_1(\bfa,\bfb,M,u,v)
		:=a_s+b_sv
		-M(a_r+b_ru)\left(1+\frac1u\right)
		\prod_{\substack{j=1\\ j\neq r}}^3
		\left(1+\frac{b_j}{a_r-a_j+b_ru}\right)
		$$ 
		and
		$$
		\Phi_2(\bfa,\bfb,M,u,v)
		:=a_r+b_ru
		-M(a_s+b_sv)\left(1+\frac1v\right)
		\prod_{\substack{j=1\\ j\neq s}}^3
		\left(1+\frac{b_j}{a_s-a_j+b_sv}\right)
		$$ 
		At $\bfb=0$ these equations become
		\[
		a_s=M(a_r)\left(1+\frac1u\right),
		\qquad
		a_r=M(a_s)\left(1+\frac1v\right),
		\]
		with solution
		\[
		u=u_{rs}^0:=\frac{M(a_r)}{a_s-M(a_r)},
		\qquad
		v=v_{rs}^0:=\frac{M(a_s)}{a_r-M(a_s)}.
		\]
		Condition~\ref{cond_3} says precisely that these two numbers are finite and nonzero. The Jacobian matrix of $(\Phi_1,\Phi_2)$ with respect to $(u,v)$, evaluated at $(\bfa, 0, M, u_{rs}^0,v_{rs}^0)$, is diagonal, with diagonal entries
		\[
		\frac{M(a_r)}{(u_{rs}^0)^2},
		\qquad
		\frac{M(a_s)}{(v_{rs}^0)^2},
		\]
		and is therefore nonsingular. The implicit function theorem now gives a cycle whose points converge to $a_r$ and $a_s$ and satisfy the asymptotic relations
		$$
		z = a_r+b_r\frac{M(a_r)}{a_s-M(a_r)} +O(\|\bfb\|),
		$$
		$$
		w = a_s+b_s\frac{M(a_s)}{a_r-M(a_s)} +O(\|\bfb\|).
		$$
		This cycle has exact period $2$ by condition~\ref{cond_1}.
		
		It remains to compute the multiplier of this $2$-cycle.
		For $j\in\{r,s\}$, set
		\[
		H_j(\zeta):=M(\zeta)
		\prod_{\substack{k=1\\ k\neq j}}^3
		\left(1+\frac{b_k}{\zeta-a_k}\right),
		\]
		so that
		\[
		G_{M,\mathbf a,\mathbf b}(\zeta)
		=H_j(\zeta)\left(1+\frac{b_j}{\zeta-a_j}\right).
		\]
		At the point $z=a_r+b_ru$, differentiation gives
		\[
		G_{M,\mathbf a,\mathbf b}'(z)
		=H_r'(z)\left(1+\frac1u\right)
		-\frac{H_r(z)}{b_ru^2}.
		\]
		The implicit functions constructed above satisfy
		\[
		u=u_{rs}^0+O(\|\bfb\|),
		\qquad
		v=v_{rs}^0+O(\|\bfb\|),
		\]
		while $H_r(z)=M(a_r)+O(\|\bfb\|)$ and $H_r'(z)=O(1)$. Consequently,
		\[
		G_{M,\mathbf a,\mathbf b}'(z)
		=-\frac{\bigl(a_s-M(a_r)\bigr)^2}{b_rM(a_r)}
		\left(1+O(\|\bfb\|)\right).
		\]
		The analogous computation at $w=a_s+b_sv$ gives
		\[
		G_{M,\mathbf a,\mathbf b}'(w)
		=-\frac{\bigl(a_r-M(a_s)\bigr)^2}{b_sM(a_s)}
		\left(1+O(\|\bfb\|)\right).
		\]
		Thus the multiplier of this cycle satisfies
		\[
		\mu_{rs}
		=G_{M,\mathbf a,\mathbf b}'(z)
		G_{M,\mathbf a,\mathbf b}'(w)
		=\frac{A_{rs}(M,\mathbf a)}{b_rb_s}
		\left(1+O(\|\bfb\|)\right),
		\]
		as required.
		
		The six cycles obtained in this way are pairwise distinct when $\mathbf b$ is sufficiently small. Indeed, their limiting unordered pairs are
		\[
		\{a_r,c_r\},\quad r\in\{1,2,3\},
		\qquad\text{and}\qquad
		\{a_r,a_s\},\quad 1\le r<s\le3,
		\]
		and these pairs are pairwise distinct by conditions~\ref{cond_1} and~\ref{cond_2}.

		Finally, all the equations used above and their limiting Jacobians depend holomorphically on $(M,\mathbf a)$. The excluded denominators and the limiting Jacobian determinants remain nonzero in a neighborhood of every parameter satisfying conditions~\ref{cond_1}--\ref{cond_3}. The implicit functions and all the resulting asymptotic estimates are therefore locally uniform in $(M,\mathbf a)$.
	\end{proof}
	
	For every pair $(r,s)$, satisfying $1\le r<s\le 3$, define
	$$
	\rho_{rs}:= \frac{\mu_{rs}}{\mu_r\mu_s} = \frac{A_{rs}(M,\bfa)}{A_r(M,\bfa)A_s(M,\bfa)}(1+O(\|\bfb\|)).
	$$

	For a future reference, let us make the following obvious statement:
	\begin{lemma}\label{lem:ratio-coordinates}
		Under the assumptions of Lemma~\ref{lem:three-hole-period-two-cycles}, for every sufficiently small $\bfb\in(\bbC^*)^3$, the multipliers $\mu_1,\mu_2,\mu_3$ are nonzero, and the change of coordinates
		\[
		(\mu_1,\mu_2,\mu_3,\mu_{12},\mu_{13},\mu_{23})
		\longmapsto
		(\mu_1,\mu_2,\mu_3,\rho_{12},\rho_{13},\rho_{23})
		\]
		is invertible in a neighborhood of the corresponding multiplier vector. Its inverse is given by
		\[
		\mu_{rs}=\rho_{rs}\mu_r\mu_s,
		\qquad 1\le r<s\le3.
		\]
		Consequently, the differentials of the six multipliers are linearly independent if and only if the differentials of $\mu_1,\mu_2,\mu_3,\rho_{12},\rho_{13},\rho_{23}$ are linearly independent.
	\end{lemma}
	
	\begin{lemma}[Limiting block form]\label{lem:limiting-block-form}
		Let $x=(x_1,x_2,x_3)$ be holomorphic parameters for a local three-parameter family $x\mapsto(M_x,\bfa_x)$ satisfying the assumptions of Lemma~\ref{lem:three-hole-period-two-cycles}, and set
		\[
		R_{rs}(x):=
		\frac{A_{rs}(M_x,\bfa_x)}
		{A_r(M_x,\bfa_x)A_s(M_x,\bfa_x)}.
		\]
		For $t\neq0$ sufficiently small, define the Jacobian matrix
		\[
		J(t,x):=
		\left.
		\frac{\partial(\mu_1,\mu_2,\mu_3,
		\rho_{12},\rho_{13},\rho_{23})}
		{\partial(b_1,b_2,b_3,x_1,x_2,x_3)}
		\right|_{b_1=b_2=b_3=t}.
		\]
		Then the matrix obtained from $J(t,x)$ by multiplying its first three rows by $t^2$ extends holomorphically to $t=0$, and its value at $t=0$ is
		\[
		\begin{pmatrix}
		-\operatorname{diag}(A_1,A_2,A_3)&0\\[0.4em]
		*&\displaystyle
		\frac{\partial(R_{12},R_{13},R_{23})}
		{\partial(x_1,x_2,x_3)}
		\end{pmatrix},
		\]
		where $A_r=A_r(M_x,\bfa_x)$ and each block is of size $3\times3$.
	\end{lemma}

	\begin{proof}
		Since multipliers of simple periodic orbits are locally analytic functions of the parameters, Lemma~\ref{lem:three-hole-period-two-cycles} shows that the functions
		\[
		\widehat\mu_r(\bfb, x):=b_r\mu_r,
		\qquad
		\widehat\mu_{rs}(\bfb,x):=b_rb_s\mu_{rs}
		\]
		extend holomorphically to $\bfb=0$. Their values there are, respectively,
		\[
		\widehat\mu_r(0,x)=A_r(M_x,\bfa_x),
		\qquad
		\widehat\mu_{rs}(0,x)=A_{rs}(M_x,\bfa_x).
		\]
		Since the functions $A_r$ are nonzero under the assumptions of Lemma~\ref{lem:three-hole-period-two-cycles}, we may shrink the parameter neighborhood so that all $\widehat\mu_r$ are nonzero. Hence
		\[
		\rho_{rs}
		=\frac{\widehat\mu_{rs}}
		{\widehat\mu_r\widehat\mu_s}
		\]
		extends holomorphically to $\bfb=0$ and satisfies $\rho_{rs}(0,x)=R_{rs}(x)$.
		
		Write the Jacobian in block form as
		\[
		J(t,x)=
		\begin{pmatrix}
		P(t,x)&Q(t,x)\\
		S(t,x)&T(t,x)
		\end{pmatrix},
		\]
		where $P,Q,S,T$ are three-by-three matrices; the first three rows of $J(t,x)$ correspond to $\mu_1,\mu_2,\mu_3$ and the first three columns to $b_1,b_2,b_3$. Since $\mu_r=\widehat\mu_r/b_r$, we have
		\[
		\frac{\partial\mu_r}{\partial b_j}
		=
		\begin{cases}
		-\dfrac{\widehat\mu_r}{b_r^2}
		+\dfrac1{b_r}\dfrac{\partial\widehat\mu_r}{\partial b_r},
		&j=r,\\[1.2ex]
		\dfrac1{b_r}\dfrac{\partial\widehat\mu_r}{\partial b_j},
		&j\neq r,
		\end{cases}
		\qquad
		\frac{\partial\mu_r}{\partial x_j}
		=\frac1{b_r}\frac{\partial\widehat\mu_r}{\partial x_j}.
		\]
		After setting $b_1=b_2=b_3=t$ and multiplying the first three rows of $J(t,x)$ by $t^2$, these identities give
		\[
		t^2P(t,x)=-\operatorname{diag}(A_1,A_2,A_3)+O(t),
		\qquad
		t^2Q(t,x)=O(t).
		\]
		In particular, the entries on the left-hand sides of both identities extend holomorphically to $t=0$.
		
		Finally, the holomorphic extension of the functions $\rho_{rs}$ implies that $S(t,x)$ and $T(t,x)$ are holomorphic at $t=0$, and
		\[
		T(0,x)=
		\frac{\partial(R_{12},R_{13},R_{23})}
		{\partial(x_1,x_2,x_3)}.
		\]
		Combining these four block limits proves the lemma.
	\end{proof}

	\begin{proof}[Proof of Proposition~\ref{lemma_period_two_degree_four}]
		Consider the three-parameter family
		\[
		M_q(z)=qz,
		\qquad
		a_1=1,\quad a_2=\xi,\quad a_3=\xi\eta,
		\]
		with parameters $(x_1,x_2,x_3)=(q,\xi,\eta)$. For this family, the coefficients from Lemma~\ref{lem:three-hole-period-two-cycles} are
		\[
		A_r=-\frac{a_r(1-q^2)^2}{q^2}
		\]
		and
		\[
		A_{rs}
		=\frac{(a_s-qa_r)^2(a_r-qa_s)^2}
		{q^2a_ra_s}.
		\]
		Consequently,
		\[
		R_{rs}
		=F\left(\frac{a_s}{a_r},q\right),
		\]
		where
		\begin{equation}\label{eq:F-residual}
		F(\zeta,q)
		:=
		\frac{q^2}{(1-q^2)^4}
		(q-\zeta)^2\left(q-\frac1\zeta\right)^2.
		\end{equation}
		In particular,
		\[
		R_{12}=F(\xi,q),
		\qquad
		R_{23}=F(\eta,q),
		\qquad
		R_{13}=F(\xi\eta,q).
		\]
		
		Differentiating~(\ref{eq:F-residual}) with respect to $\zeta$, we obtain
		\[
		F_\zeta(\zeta,q)
		=
		-\frac{2q^3}{(1-q^2)^4}
		(q-\zeta)\left(q-\frac1\zeta\right)
		\left(1-\frac1{\zeta^2}\right).
		\]
		Thus $F_\zeta(-1,q)=0$, for any $q\neq \pm 1$. If we order the functions as $(R_{23},R_{12},R_{13})$ and the parameters as $(q,\xi,\eta)$, then at $\eta=-1$ the residual Jacobian becomes the lower triangular matrix
		\[
		\frac{\partial(R_{23},R_{12},R_{13})}
		{\partial(q,\xi,\eta)}
		=
		\begin{pmatrix}
		F_q(-1,q)&0&0\\[0.4em]
		F_q(\xi,q)&F_\zeta(\xi,q)&0\\[0.4em]
		F_q(-\xi,q)&-F_\zeta(-\xi,q)&
		\xi F_\zeta(-\xi,q)
		\end{pmatrix},
		\]
		where $F_\zeta$ and $F_q$ denote the corresponding partial derivatives of the function $F$ from~(\ref{eq:F-residual}).

		We now choose
		\begin{equation}\label{par_values_eq}
		q=2,
		\qquad
		\xi=3,
		\qquad
		\eta=-1.
		\end{equation}
		This gives $M(z)=2z$ and $(a_1,a_2,a_3)=(1,3,-3)$. A direct computation shows that these parameters satisfy the assumptions of Lemma~\ref{lem:three-hole-period-two-cycles}. Moreover, we get
		\[
		F(-1,q)=\frac{q^2}{(1-q)^4},
		\qquad
		F_q(-1,q)=\frac{2q(1+q)}{(1-q)^5},
		\]
		so $F_q(-1,2)\neq0$. The displayed formula for $F_\zeta$ also gives
		\[
		F_\zeta(3,2)\neq0,
		\qquad
		F_\zeta(-3,2)\neq0.
		\]
		Hence all three diagonal entries of the residual Jacobian are nonzero. Permuting its rows or columns does not affect nonsingularity, and therefore
		\[
		\det
		\frac{\partial(R_{12},R_{13},R_{23})}
		{\partial(q,\xi,\eta)}
		\neq0.
		\]
		At the same parameter values~(\ref{par_values_eq}), we also get
		\[
		(A_1,A_2,A_3)
		=
		\left(-\frac94,-\frac{27}{4},\frac{27}{4}\right),
		\]
		so $A_1A_2A_3\neq0$. Lemma~\ref{lem:limiting-block-form} now implies that, for every sufficiently small nonzero $t$, the functions
		\[
		\mu_1,\mu_2,\mu_3,
		\rho_{12},\rho_{13},\rho_{23}
		\]
		have linearly independent differentials in the family with parameters $(b_1,b_2,b_3,q,\xi,\eta)$ at $b_1=b_2=b_3=t$ and $q,\xi,\eta$ from~(\ref{par_values_eq}). By Lemma~\ref{lem:ratio-coordinates}, the same holds for the six period-two multipliers
		\[
		\mu_1,\mu_2,\mu_3,
		\mu_{12},\mu_{13},\mu_{23}
		\]
		of the corresponding degree $4$ rational map $G_{M_q,\bfa,\bfb}$, provided by Lemma~\ref{lem:three-hole-period-two-cycles}. Furthermore, according to Lemma~\ref{lem:three-hole-period-two-cycles}, when $t$ is sufficiently small, these six period two cycles are distinct, hence, non-multiple, since degree $4$ rational maps cannot have more than $6$ period two cycles. Proposition~\ref{prop_local_to_global} therefore implies that the six period-two multipliers on $\Rat_4$ are algebraically independent over $\bbC$.
	\end{proof}

\bibliographystyle{amsalpha}
\bibliography{paper_v1}

\end{document}